\documentclass{amsart}

\usepackage{amsmath}
\usepackage{amssymb}
\usepackage{amsthm}
\usepackage{array}
\usepackage{subfigure}
\usepackage[all]{xy}
\usepackage{graphics,graphicx}

\DeclareFontFamily{OT1}{rsfs}{}
\DeclareFontShape{OT1}{rsfs}{n}{it}{<-> rsfs10}{}
\DeclareMathAlphabet{\mathscr}{OT1}{rsfs}{n}{it}

\newcommand{\ZZ}{\mathbb Z}
\newcommand{\RR}{\mathbb R}
\newcommand{\QQ}{\mathbb Q}

\newcommand{\NN}{\mathbb N}
\newcommand{\gs}{\geqslant}
\newcommand{\ls}{\leqslant}
\newcommand{\zero}{\boldsymbol{0}}

\newcommand{\conv}{\operatorname{conv}}
\newcommand{\gr}{\operatorname{gr}}
\newcommand{\HH}[3]{\operatorname{H}^{#1}_{#2}(#3)}
\newcommand{\ver}{\operatorname{vert}}
\newcommand{\cone}{\operatorname{pyr}}
\newcommand{\pyr}{\cone}
\newcommand{\out}{\operatorname{out}}
\newcommand{\conen}[1]{\operatorname{cone}_{#1}}
\newcommand{\bd}{\operatorname{bd}}
\newcommand{\vol}{\operatorname{vol}}

\newcommand{\iso}{\cong}
\newcommand{\polyjoin}{\amalg}
\newcommand{\nr}{w}
\newcommand{\tope}{\mcal{F}}
\newtheorem{thm}{Theorem}[section]

\newtheorem{prop}[thm]{Proposition}
\newtheorem{lem}[thm]{Lemma}
\theoremstyle{definition}

\newtheorem{xmp}[thm]{Example}
\newtheorem{rmk}[thm]{Remark}

\def\fm{\mathfrak{m}}
\def\mcal{\mathcal}

\title[The $j$-Multiplicity of Monomial Ideals]{The $j$-Multiplicity of Monomial Ideals}
\author[Jack Jeffries]{Jack Jeffries}
\address{Department of Mathematics \\ University of Utah \\ 155 S 1400 E 
Salt Lake City, UT. 84112}
\email{jeffries@math.utah.edu}
\author[Jonathan Monta\~{n}o]{Jonathan Monta\~{n}o}
\address{Department of Mathematics \\ Purdue University \\150 North University Street West Lafayette, IN. 47907}
\email{jmontano@math.purdue.edu}

\begin{document}
\thanks{This material is based upon work supported by the National Science Foundation under Grant No~0932078 000, while the authors were in residence at the Mathematical Science Research Institute in Berkeley, California, during the Fall semester of 2012. The first author was also supported in part by NSF grants DMS~0758474 and DMS~1162585, and the second author by NSF grant DMS~0901613.}
\begin{abstract} We prove a characterization of the $j$-multiplicity of a monomial ideal as the normalized volume of a polytopal complex. Our result is an extension of Teissier's volume-theoretic interpretation of the Hilbert-Samuel multiplicity for $\fm$-primary monomial ideals. We also give a description of the $\varepsilon$-multiplicity of a monomial ideal in terms of the volume of a region.
\end{abstract}

\maketitle

\section{Introduction}
The $j$-multiplicity was defined in 1993 by Achilles and Manaresi in \cite{AM} as a generalization of the Hilbert-Samuel multiplicity for arbitrary ideals in a Noetherian local ring. Several results on the Hilbert-Samuel multiplicity have been successfully extended to more general classes of ideals using the $j$-multiplicity, for example \cite{FM1}, \cite{PX}, and \cite{Ciu}. The main result of this paper may be viewed as one of these extensions. 

Let $(R,\fm,k)$ be a Noetherian local ring of dimension $d$, and $I\subset R$ an ideal. The $j$-multiplicity of $I$ is defined as the limit 
\[j(I)=\lim_{n\rightarrow \infty}\frac{(d-1)!}{n^{d-1}}\lambda_R\big(\HH{0}{\fm}{I^n/I^{n+1}}\big)\,.\]

There have been previous approaches for computing the $j$-multiplicity. For example, in \cite{AM} and \cite{Xie} it is proven that if $k$ is infinite, then for general elements $a_1,\ldots,a_{d}$ in $I$, and $\alpha=(a_1,\ldots,a_{d-1})$, we have \[j(I)=\lambda_R \big(R/((\alpha:_R\,I^{\infty})+a_d R)\big)\,.\]
This formula is applied to compute specific examples in \cite{NU}.

Let $R$ denote now the polynomial ring $k[x_1,\ldots,x_d]$ over the field $k$, $\fm$ the homogeneous maximal ideal $(x_1,\ldots,x_d)$, and $I$ a monomial ideal of $R$. The {\it Newton polyhedron} of $I$ is the convex hull of the points in $\RR^d$ that correspond to monomials in $I$, which we will denote by $\conv(I)$. In this paper we generalize the classical result that describes the Hilbert-Samuel multiplicity of an $\fm$-primary ideal as the normalized volume of the complement of its Newton polyhedron in $\RR^d_{\gs 0}$, see \cite{Tes}. If $I$ is not $\fm$-primary, the complement of $\conv(I)$ is infinite, but we can define the analogue of this region in the general case by considering the truncated cone from the origin to the union of the bounded faces of $\conv(I)$. This truncated cone will be denoted by $\cone(I)$. With this notation, we can state our main result:

\

{\bf Theorem \ref{main}.} {\it Let $I\subset R$ be a monomial ideal. Then $j(I)=d!\vol(\cone(I))$.}

\

Earlier unpublished work of J. Validashti obtains this formula in dimension two. The rest of the paper is organized as follows: In the second section we set up the notation and also present some results that will be used in the proof of the main theorem. The third section will include the proof of Theorem~\ref{main}. In the fourth section we provide an extension of this result to pointed normal affine toric varieties. In the fifth section we will apply our characterization of the saturation of a monomial ideal in $R$ in Lemma~\ref{satur} to give a geometric description of the $\varepsilon$-multiplicity. The paper ends with some examples in a sixth section.

\section{Preliminaries}
Let $R=k[x_1,\ldots,x_d]$ be a polynomial ring over a field $k$ and $\fm=(x_1,\ldots,x_d)$ its homogeneous maximal ideal. Let $I$ be a monomial ideal of $R$ minimally generated by $x^{v_1},x^{v_2},\ldots, x^{v_n}$ where $v_i=(v_{i,1},\ldots,v_{i,d})$ and $x^{v_i}=x_1^{v_{i,1}}\cdots x_d^{v_{i,d}}$. For a monomial ideal $L$ in $R$ we denote by $\Gamma(L)$ the set of lattice points in $\RR^d$ corresponding to the exponents. Additionally, if $L_1\supseteq L_2$ are monomial ideals, we will write $\Gamma(L_1/L_2)$ for $\Gamma(L_1)\setminus\Gamma(L_2)$.

We denote by $\conv(I)$ the {\it Newton polyhedron} of $I$, that is:
\[\conv(I):= \conv(v_1,\ldots, v_n)+\mathbb{R}^d_{\gs 0}\, , \]
where $+$ denotes the Minkowski sum. It is worth noting that the collection of bounded facets of the Newton polyhedron is not convex, and thus is not a polytope, but rather has the structure of a polytopal complex. Notice also that $\conv(I)=\conv(\Gamma(I))$. Since every polyhedron is defined by the intersection of finitely many closed half spaces, we can define $H_i=\{x\in \mathbb{R}^n\mid \langle x,b_i\rangle = c_i\}$, with $b_i\in \mathbb{Q}^d$, $c_i\in \mathbb{Q}$ for $i=1,\ldots, \nr$ to be the {\it supporting hyperplanes} of $\conv(I)$ such that 
\[\conv(I)= H_{1}^+\cap H_{2}^+\cap\cdots \cap H_{\nr}^+ \, ,\] 
where $H_i^+=\{x\in \mathbb{R}^n\mid \langle x,b_i\rangle \gs c_i\}$. Let $\mcal{F}_i=H_i\cap \conv(I)$ for $i=1,\ldots, w$ be the facets of $\conv(I)$. We will assume that $H_1,\ldots,H_u$, are the hyperplanes corresponding to unbounded facets.

It can be shown that all the vectors $b_i$ have nonnegative components, and that $b_i\in \mathbb{R}^d_{>0}$ if and only if $\mcal{F}_i$ is a bounded facet, as in \cite[Lemma~1.1]{Sin}. This forces the $c_i$ to be nonnegative, and in fact positive in the case of a bounded facet.

Recall that the {\it analytic spread} of an arbitrary ideal $I$, denoted by $l(I)$, is defined to be the dimension of its {\it special fiber ring} $\gr_I (R)\otimes_R R/\fm=\bigoplus_{n=0}^{\infty} I^n/\fm I^{n},$ where $\gr_I(R)$ is the associated graded algebra of $I$, i.e., $\gr_I(R)=\bigoplus_{n=0}^{\infty} I^n/I^{n+1}$. We will say that $I$ has \emph{maximal analytic spread} if $l(I)=\dim(R)$. If $I$ is monomial, $l(I)$ can be computed as $c+1$ where $c$ is the highest dimension of a bounded facet of $\conv(I)$, see \cite[Theorem~2.3]{Biv} or \cite[Corollary~4.10]{Sin}.  

We denote by $\ver(I)$ the set of vertices of $\conv(I)$, and set $\bd(I)=\bigcup_{i=u+1}^{\nr}\mcal{F}_i$ for the union of the bounded facets of $\conv(I)$.  If $\mcal{P}$ is a polytope, we will write $\cone(\mcal{P})$ for $\conv(\mcal{P},\zero)$, the truncated cone, or pyramid, over $\mcal{P}$. By abuse of notation, we will write $\cone(I)$ for $\bigcup_{i=u+1}^{\nr}\cone(\mcal{F}_i)$. Note that the monomials corresponding to the points in $\ver(I)$ are part of the set of minimal generators of $I$, so we will assume $\ver(I)=\{v_1,\ldots,v_s\}$ for some $1\ls s\ls n$. 
We will also find it convenient to define the \emph{$n^{\text{th}}$ cone section} of a polytope $\mcal{P}$ as
\[\conen{n}(\mcal{P}):=\big((n+1)\cone(\mcal{P})\setminus (n+1)\mcal{P}\big) \, \setminus \, \big(n\cone(\mcal{P})\setminus n\mcal{P}\big) \, ,
\]
which we may alternatively write as $\bigcup_{n\ls s < n+1} s\mcal{P}$. We again abuse notation by writing $\conen{n}(I)$ for  $\bigcup_{i=u+1}^{\nr}\conen{n}(\mcal{F}_i)$.

\begin{figure}
\centering
\mbox{\subfigure{\includegraphics[width=2in]{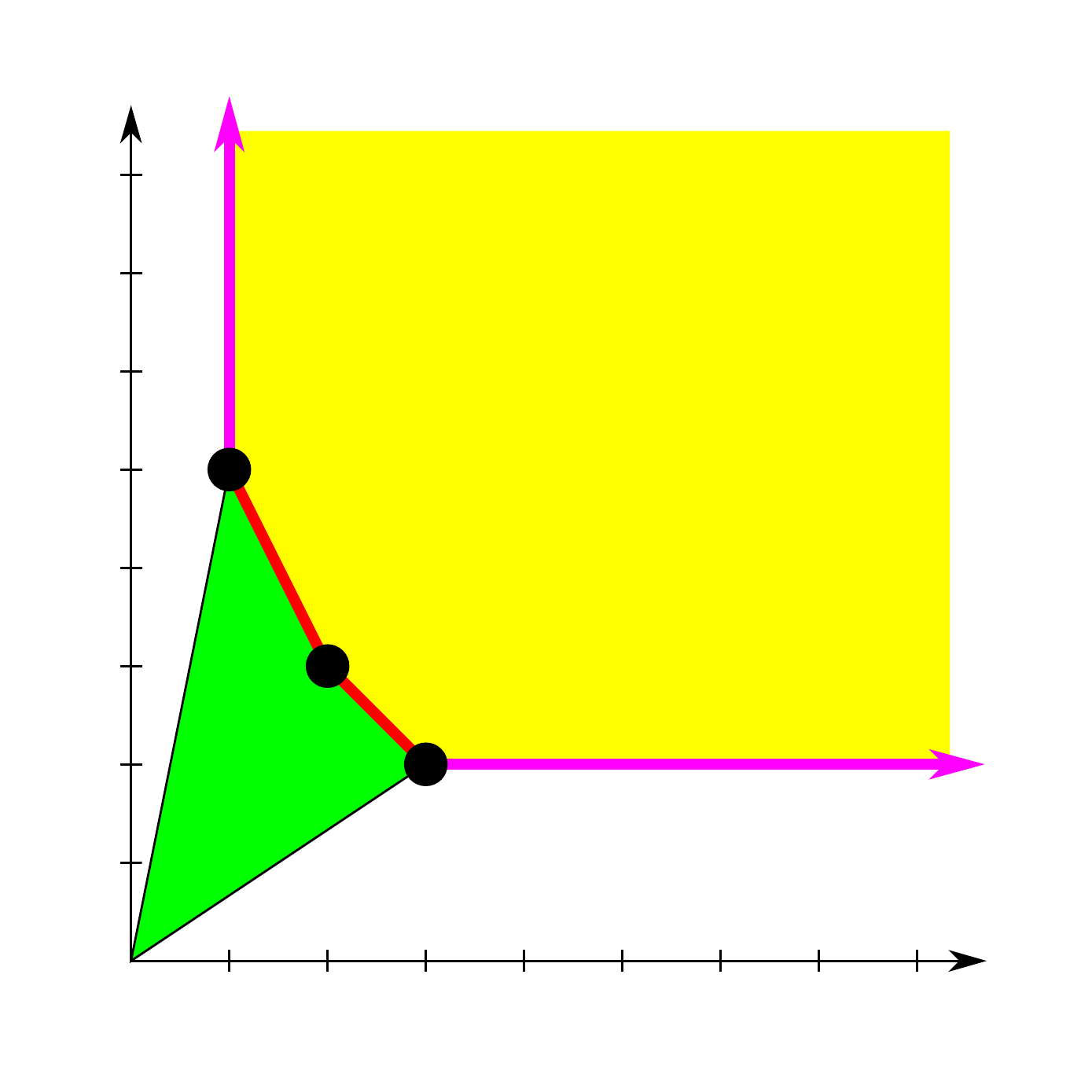}}\quad
\subfigure{\includegraphics[width=2in]{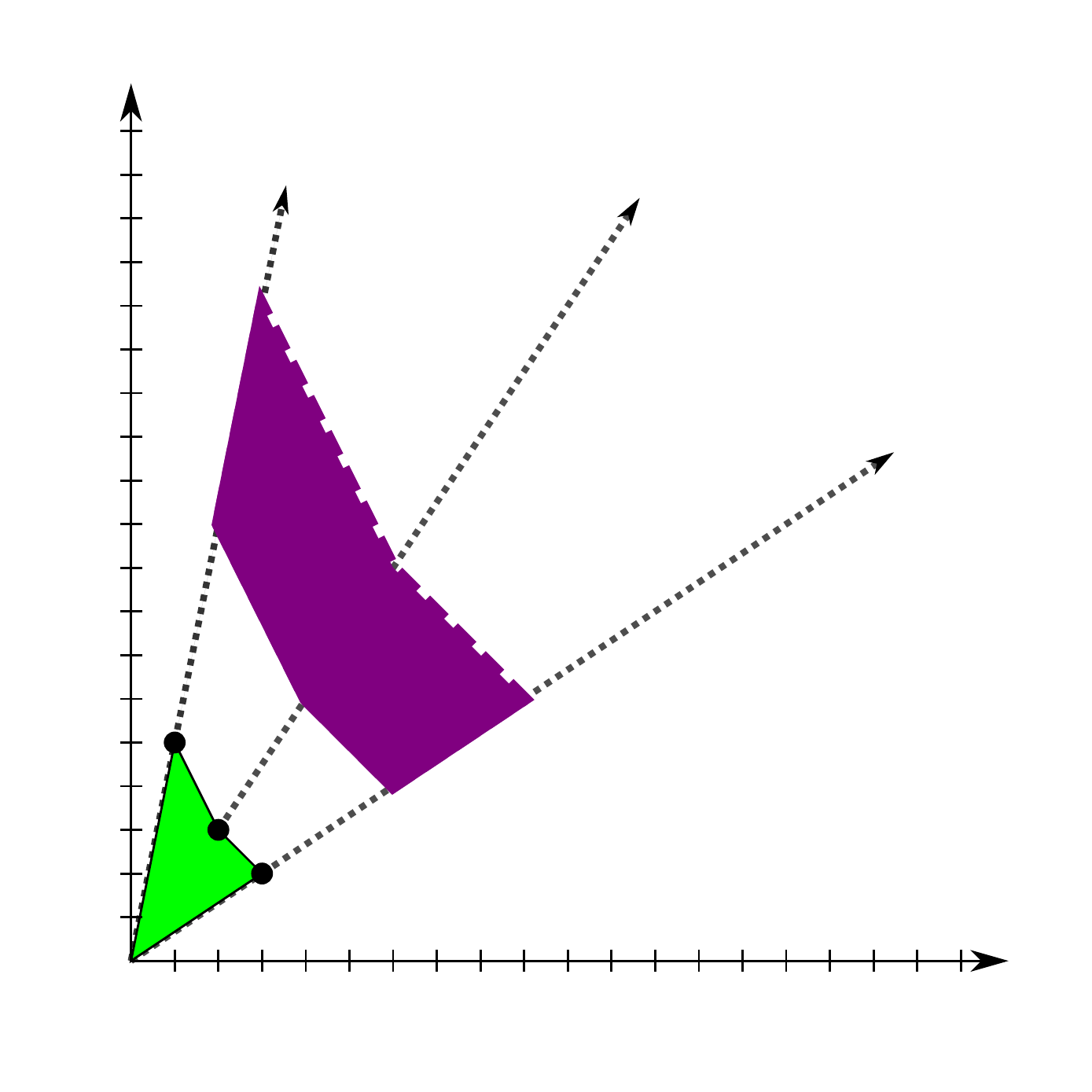} }}
\caption{Various regions for the monomial ideal $I=(xy^5,x^2y^3,x^3y^2)$} \label{fig1}
\end{figure}

For the monomial ideal $I=(xy^5,x^2y^3,x^3y^2)$, in Figure~\ref{fig1}, on the left we mark $\ver(I)$ with black dots, $\bd(I)$ with dark red lines, and the unbounded facets with pink lines. In this example, the green region with its boundary forms $\cone(I)$, and the yellow region with its boundary forms $\conv(I)$. In the graph on the right, we shade $\conen{2}(I)$ in purple, where the top dotted segments of the boundary are not included.

The following description of the faces of the Newton polyhedron will be useful.
\begin{lem}\cite[Lemma~3.1]{Sin} \label{faces}
Let $\mcal{F}$ be a face of $\conv(I)$ with supporting hyperplane
$H=\{x\in \mathbb{R}^n\mid \langle x,b\rangle=c\}$. Then $\mcal{F}\cap \ver(I)=\{v_{i_1},\ldots, v_{i_r}\}$ is non-empty, and 
\[\mcal{F}=\conv(v_{i_1},\ldots, v_{i_r})+\sum_{j : b_j=0}\mathbb{R}_{\gs 0}\, e_j\] where $e_j$ is the unit vector with nonzero $j^{\text{th}}$ component. 

\end{lem}

\

Recall that for any submodule $N$ of an $R$-module $M$, the {\it saturation} of $N$, denoted $(N:_M \fm^{\infty})$, is the set of elements $a$ in $M$ for which there exists $n\in \NN$ such that $a\fm^n\in N$. The \emph{zeroth local cohomology module} of $M$ is defined to be $(0:_M\,\fm^{\infty})$ and is denoted by $\HH{0}{\fm}{M}$. Notice that in general, $\HH{0}{\fm}{M/N}=(N :_M \fm^{\infty})/N$. If $I$ is a monomial ideal, then $(I:_R\,\fm^{\infty})$ is also monomial.
\  

The {\it integral closure} of an arbitrary ideal $J$ is the set of elements $x$ in $R$ that satisfy an integral relation $x^n+a_1 x^{n-1}+\cdots+a_{n-1}x+a_n$ where $a_i\in J^i$ for $i=1,\ldots, n$. It is denoted by $\overline{J}$ and it is an ideal.
For monomial ideals, it is possible to give a geometric description of the integral closure, namely $\Gamma(\overline{I})=\ZZ^d\cap \conv(I)$, i.e., $\conv(\overline{I})=\conv(I)$; see \cite[Proposition~7.25]{Vas}.

\
  
\begin{prop}\label{satur}
 
$\displaystyle{\Gamma(\overline{I} :_R \,m^\infty)=H_1^+\cap\cdots \cap H_u^+\cap \ZZ^d_{\gs 0}}$\,.
\end{prop}
\begin{proof}

Let $v\in H_1^+\cap\cdots \cap H_u^+\cap \ZZ^d_{\gs 0}$; then $\langle v, b_i\rangle\gs c_i$ for $i=1,\ldots, u$. For $t\in \RR_{\gs 0}$, one has \[\langle v+te_j, b_i\rangle\gs c_i+tb_{i,j}\gs c_i\] for $j=1,\ldots, d$ and $i=1,\ldots, u$. Also, if $u+1\ls i\ls \nr$ then all the entries of $b_i$ are positive, so we also have that $\langle v+te_j, b_i\rangle \gs c_i$ for $t\gg 0$. We conclude $x_j^tx^v\in \overline{I}$ 
for $j=1,\ldots, d$ and $t\gg 0$, that is $v\in \Gamma(\overline{I} :_R\, m^\infty)$.

Conversely, if $v\in \Gamma(\overline{I} :_R \,m^{\infty})$, then $v+te_j\in \conv(I)\subset H_1^+\cap\cdots \cap H_u^+$ for 
$j=1,\ldots, d$ and $t\gg 0$. Now, suppose $v\not\in H_i^+$ for some $1\ls i\ls u$, then $\langle v, b_i\rangle<c_i$. By Lemma~\ref{faces}, since $\mcal{F}_i$ is an 
unbounded facet, we can pick $j$ such that $b_{i,j}=0$, and hence $\langle v+te_j, b_i\rangle=\langle v, b_i\rangle < c_i$ for every $t\in \RR$, which is a contradiction.
\end{proof}

\begin{lem}\label{satincone}
 $\displaystyle{\Gamma(\overline{I} :_R \,\fm^{\infty})\subset \conv(I)\cup \cone(I)\,.}$
\end{lem}

\begin{proof}

The result follows immediately by Proposition \ref{satur} if $\conv(I)$ does not have bounded facets. We will assume that $u<\nr$.
 
Let $v$ be a nonzero vector in $\ZZ^d_{\gs 0} \setminus (\cone(I)\cup
\conv(I))$. We will proceed by contradiction. Suppose $v\in
\Gamma(\overline{I} :_R\, m^\infty)$, then $v\in H_{1}^+\cap\cdots
\cap H_u^+$ by Proposition~\ref{satur}. 

Note that since $b_i$ has positive entries for each $i=u+1,\ldots, {\nr}$,
we have $\langle v,b_i\rangle\neq 0$. For each $u+1\ls
i\ls {\nr}$ we can find a real number $t_i$ such that $t_i v\in H_i$, each of which is positive because $c_i>0$. Suppose, without
loss of generality, that $t_{\nr}$ is largest among $t_{u+1},\ldots, t_{\nr}$. Since $v\not\in
H_i^+$ for some $u+1 \ls i\ls {\nr}$ we have $\langle v,
b_i\rangle<c_i$, then $t_i>1$, and so $t_{\nr}>1$.

Now, $\langle t_{\nr} v , b_i\rangle\gs\langle t_i v , b_i\rangle= c_i $ for $i=u+1,\ldots, {\nr}$, so
\[
t_{\nr} v\in H_{1}^+\cap H_2^+\cap\cdots\cap H_{{\nr}-1}^+\cap H_{\nr}=\mathcal{F}_{\nr}\,.\]
Then we have $v\in\cone(I)$ which is a contradiction.
\end{proof}


In the following lemma, we will use the notion of the \emph{Hausdorff distance} between compact sets $A$ and $B$ in $\RR^d$, which is defined as
\[\rho(A,B):=\inf\{\lambda\gs 0 \,|\, A\subseteq B+\lambda U \,,\, B\subseteq A+\lambda U\}\,,\]
where $U$ is the unit ball. We will use a related notion for polytopes: for a convex polytope $\mcal{P}=\conv(v_1,\dots,v_t)$ in $\RR^d$, we will say that another convex polytope with $t$ vertices $\mcal{P}'=\conv(v'_1,\dots,v'_t)$ is an \emph{$\varepsilon$-shaking} of $\mcal{P}$ if $|v_j-v'_j|<\varepsilon$ for all $j$.

\begin{lem}\label{close} Fix $\mcal{P}=\conv(v_1,\dots,v_t)$ in $\RR^d$. Let $(\mcal{P}^{(n)}_1,\dots,\mcal{P}^{(n)}_s)_{n\in \NN}$ be a sequence of $s$-tuples of polytopes such that each $\mcal{P}^{(n)}_j$ is a $(1/n)$-shaking of $\mcal{P}$. Then,
\[\lim_{n\rightarrow\infty} \vol \left( \mcal{P}\cap\bigcap_{i=1}^s \mcal{P}^{(n)}_i \right) = \vol(\mcal{P})\,.\]
\end{lem}
\begin{proof} 

Let $\mcal{P}'=\conv(v'_1,\dots,v'_t)$ be an $\varepsilon$-shaking of $\mcal{P}$, and write 
\[
\mcal{P}'\polyjoin\mcal{P}=\conv(v'_1,\dots,v'_t,v_1,\dots,v_t)\,.
\]
Note that $\mcal{P}'\cup\mcal{P}\subseteq \mcal{P}'\polyjoin\mcal{P}$. Also, $\rho(\mcal{P},\mcal{P}\polyjoin\mcal{P}')< \varepsilon$, since for $q\in \mcal{P}\polyjoin\mcal{P}'$, we may write $q=\sum{\lambda_j} v_j+\sum{\lambda'_j} v'_j$ with $\sum\lambda_i+\sum\lambda'_i=1$, and
\[|q- \sum{\lambda_j} v_j-\sum{\lambda'_j} v_j| \,=\, |\sum{\lambda'_j} v'_j- \sum{\lambda'_j} v_j| \,\ls\, \sum{\lambda'_j} |v'_j- v_j| \,<\, \varepsilon\,,\]
where $\sum{\lambda_j} v_j+\sum{\lambda'_j} v_j\in \mcal{P}$. 
Similarly $\rho(\mcal{P},\mcal{P}')< \varepsilon$.

We have
\begin{align*}
0\ls \vol(\mcal{P})-\vol(\mcal{P}\cap\bigcap_{i=1}^s \mcal{P}^{(n)}_i)   
&\ls \sum_{i=1}^s {\big(\vol(\mcal{P})-\vol(\mcal{P}\cap\mcal{P}^{(n)}_i)\big)}  \\
&= \sum_{i=1}^s {\big(\vol(\mcal{P}\cup\mcal{P}^{(n)}_i)-\vol(\mcal{P}^{(n)}_i)\big)}   \\
 &\ls \sum_{i=1}^s {\big(\vol(\mcal{P}\polyjoin\mcal{P}^{(n)}_i)-\vol(\mcal{P}^{(n)}_i)\big)}\,.
\end{align*}

Then, by continuity of volume with respect to Hausdorff distance, see \cite[Theorem~6.2.17]{Web}, we have that $\vol(\mcal{P}^{(n)}_i)$ and $\vol(\mcal{P}\polyjoin\mcal{P}^{(n)}_i)$ both converge to $\vol(\mcal{P})$ as $n\rightarrow \infty$. Thus, $\vol(\mcal{P}\cap\bigcap_{i=1}^s \mcal{P}^{(n)}_i)\rightarrow \vol(\mcal{P})$ as $n\rightarrow \infty$, as required.
\end{proof}

Recall that the \emph{Ehrhart function} of a polytope $\mcal{P}\subset \RR^d$ is defined as
\[ E_{\mcal{P}}(n):=\#(\ZZ^d \cap n\mcal{P})\,. \]
Ehrhart \cite{Ehr} showed that if the vertices of $\mcal{P}$ have integer coordinates, $E_{\mcal{P}}(n)$ is a polynomial of degree $\dim(\mcal{P})$ with leading coefficient equal to the relative volume of $\mcal{P}$ (cf., \cite[Chapter~12]{MS}). We will employ a strengthening of this fact. Recall that a function $f\colon\NN\rightarrow \ZZ$ is called a \emph{quasi-polynomial} if there is an $m\in \NN$ and polynomials $f_0,\dots,f_{m-1}$ such that $f(n)=f_{(n\, \text{mod}\, m)}(n)$ for all $n\in \NN$. The \emph{grade} of $f$ is the least $\delta$ such that the $i^{\text{th}}$ coefficient of each of the $f_j$ is the same for all $i>\delta$.
The following was conjectured by Ehrhart \cite{Ehr}, and proved by McMullen \cite{McM} and Stanley \cite{Sta} separately (see also \cite{BI} for a proof based on monomial ideal techniques).

\begin{thm}[McMullen, Stanley]\label{MS} Let $\mcal{P} \subset \RR^d$ be a polytope with vertices in $\QQ^d$. Suppose that the affine span of each $t$-dimensional face of $\mcal{P}$ contains a lattice point. Then $\# (\ZZ^d \cap n\mcal{P})$ as a function of $n$ is given by a quasi-polynomial of degree equal to the dimension of $\mcal{P}$ and grade less than $t$.
\end{thm}

\begin{prop}\label{poly} $\phantom{M}$

\begin{enumerate}
\item[a)] Let $\mcal{P} \subset \RR^d$ be a polytope with vertices in $\QQ^d$, and $\dim \mcal{P}<d$. Suppose that the affine span of $\mcal{P}$ contains a point in the integer lattice $\ZZ^d$, or that the dimension of $\mcal{P}$ is less than $d-1$.
Then $\# (\ZZ^d \cap \conen{n}(\mcal{P}))$ as a function of $n$ is given by a quasi-polynomial of the form 
\[
p(n)=d\vol(\cone(\mcal{P}))\, n^{d-1} + O(n^{d-2})\,.
\]
\item[b)] If $I$ is a monomial ideal, then $\# (\ZZ^d \cap \conen{n}(I))$ is given by a quasi-polynomial of the form
 \[p(n)=d\vol(\cone(I))\, n^{d-1} + O(n^{d-2})\,.\]
\end{enumerate}
\end{prop}

\begin{proof} Write
\begin{align*} \# (\ZZ^d \cap \conen{n}(\mcal{P})) &= \# \bigg(\ZZ^d \cap  \Big(\big((n+1)\cone(\mcal{P})\setminus (n+1)\mcal{P}\big)\setminus \big(n\cone(\mcal{P})\setminus n\mcal{P}\big)\Big) \bigg)  \\
&= \big(E_{\cone(\mcal{P})}(n+1)-E_{\cone(\mcal{P})}(n)\big)-\big(E_{\mcal{P}}(n+1)-E_{\mcal{P}}(n)\big)\,.
\end{align*}
Notice that the hypothesis in part a) ensures the affine span of each $(d-1)$-dimensional face of $\cone(\mcal{P})$ contains a lattice point: if $\dim \mcal{P}=d-1$, because the affine span of $\mcal{P}$ has a lattice point and every other $(d-1)$-dimensional face contains $\zero$. If $\dim\mcal{P}=d-2$, the only $(d-1)$-dimensional face contains $\zero$, and the condition is vacuous otherwise. By Theorem~\ref{MS}, in this situation $E_{\cone(\mcal{P})}$ and $E_{\mcal{P}}$ are quasi-polynomials of the form
\begin{align*}
E_{\cone(\mcal{P})}(n) &=  a_{d} n^d + a_{d-1} n^{d-1} + O(n^{d-2}) \\
E_{\mcal{P}}(n) &= b_{d-1} n^{d-1} + O(n^{d-2})\,,
\end{align*}
where $a_d$, $a_{d-1}$, and $b_{d-1}$ are constants; specifically, they do not depend on $n$. Further, from the definition of the Riemann integral we compute that 
\[a_d=\lim_{n\rightarrow\infty} n^{-d}\cdot\#\big( (\tfrac{1}{n}\ZZ)^d \cap \cone(\mcal{P})\big) = \vol(\cone(\mcal{P}))\,.\]
Part a) now follows from the formula above. 

For part b), we first show that that for two different bounded faces $\tope$, $\tope'$ of  $\conv(I)$ we have $\conen{n}(\tope)\cap \conen{n}(\tope')=\conen{n}(\tope\cap \tope')$. Indeed, let $v$ be a nonzero element of $\pyr(\tope)\cap \pyr(\tope)$  and $t, r\gs 1$ such that $tv\in \tope$ and $rv\in \tope'$. If $t\gs r $, then $\frac{t}{r}(rv)=tv\in \tope$ implies that $tv\in rv+\RR^d_{\gs 0}$, but $tv$ is on the boundary of $\conv(I)$ so $t=r$ and $v\in \pyr(\tope\cap \tope')$. Now, the claim follows from the definition of $\conen{n}$ and the fact that $n\mcal{P}\cap n\mcal{P}'=n(\mcal{P}\cap \mcal{P}')$ for any pair of polytopes $\mcal{P}$, $\mcal{P}'$.
Then by inclusion-exclusion we have
 \begin{align*}
\# (\ZZ^d \cap \conen{n}(I))=& \sum_{i=u+1}^{w}  \# (\ZZ^d \cap \conen{n}(\mcal{F}_i))-\!\!\!\!\!\!\sum_{u+1\leq i<j\leq w}\!\!\!\!\!\!\# (\ZZ^d \cap \conen{n}(\mcal{F}_i\cap \mcal{F}_j))\\ &+\,\cdots\, \pm \# (\ZZ^d \cap \conen{n}(\mcal{F}_{u+1}\cap \cdots \cap \mcal{F}_w))\,.
\end{align*}
The conclusion now follows from part a).
\end{proof}

\section{The $j$-multiplicity of monomial ideals}

In order to be consistent with the definition of $j$-multiplicity, which is defined for ideals in a local ring, in this chapter we will consider $R=k[x_1,\ldots,x_d]_{\fm}$ and $I$ an ideal generated by monomials. All the results of the second section still hold in this setting, because all the ideals involved are monomial ideals. Moreover, the analytic spread does not change.

For an $R$-module $M$, we can define the {\it $j$-multiplicity} of $I$ with respect to $M$ as 
\[j(I,M)=\lim_{n\rightarrow \infty}\frac{(d-1)!}{n^{d-1}}\lambda_R\big(\HH{0}{\fm}{I^nM/I^{n+1}M}\big)\,.\]
In the case $M=R$, $j(I,R)$ will be denoted $j(I)$ as in the introduction.

\

The following proposition shows that we can compute $j(I)$ using the filtration $\{\overline{I^n}\}_{n\in\NN}$. The proof is similar to \cite[Proposition~2.10]{FM1}.
\begin{prop}\label{jbar} Let $I$ be a monomial ideal, then
\[\displaystyle{j(I)=\lim_{n\rightarrow\infty}\frac{(d-1)!}{n^{d-1}}\lambda_R\big(\HH{0}{\fm}{\overline{I^n}/\overline{I^{n+1}}}\big)}\,.\]
\end{prop}

\begin{proof}
By \cite[Theorem~7.29]{Vas}, $\overline{I^{n+1}}=I\overline{I^n}$ for
$n\gs d$. From \cite[Theorem~3.11]{NU}, and the following exact sequence of $R$-modules 

\[0\rightarrow \overline{I^d}\rightarrow R \rightarrow
R/\overline{I^d}\rightarrow 0 \,,\]
we obtain  $j(I,R)=j(I,\overline{I^d})+j(I,R/\overline{I^d})$.
But $I^d(R/\overline{I^d})=0$, so
$j(I,R/\overline{I^d})=0$ and then $j(I)=j(I,R)=j(I,\overline{I^d}).$

Now,
\begin{align*}
j(I)=j(I,\overline{I^d})&=\lim_{t\rightarrow\infty}\frac{(d-1)!}{t^{d-1}}\lambda_R\big(\HH{0}{\fm}{I^{t}\overline{I^{d}}/I^{t+1}\overline{I^{d}}}\big)\\
&= \lim_{t\rightarrow\infty}\frac{(d-1)!}{t^{d-1}}\lambda_R\big(\HH{0}{\fm}{\overline{I^{t+d}}/\overline{I^{t+d+1}}}\big)\\
&= \lim_{n\rightarrow\infty}\frac{n^{d-1}}{(n-d)^{d-1}} \frac{(d-1)!}{n^{d-1}}\lambda_R\big(\HH{0}{\fm}{\overline{I^{n}}/\overline{I^{n+1}}}\big)\,,
\end{align*}
and the result follows from the equality $\displaystyle{\lim_{n\rightarrow\infty}
\frac{n^{d-1}}{(n-d)^{d-1}}=1}$.
\end{proof}

\begin{thm}\label{main} Let $I\subset R$ be a monomial ideal. Then $j(I)=d!\vol(\cone(I))$.
\end{thm}
\begin{proof}
By Lemma~\ref{jbar}, we compute 
\[
j(I)=\lim_{n\rightarrow\infty}\frac{(d-1)!}{n^{d-1}}\lambda_R\big(\HH{0}{\fm}{\overline{I^{n}}/{\overline{I^{n+1}}}}\big)\,.
\]

\subsection*{Step 1:} The proof that $j(I)\ls d!\vol(\cone(I))$:

Recall that $\HH{0}{\fm}{\overline{I^{n}}/\overline{I^{n+1}}}=\big((\overline{I^{n+1}}: \fm^{\infty})\cap \overline{I^{n}}\big) /\overline{I^{n+1}}$, so that 

\begin{align*}
\lambda_R\big(\HH{0}{\fm}{\overline{I^{n}}/\overline{I^{n+1}}}\big)&=\#\Big(\ZZ^d\cap \big(\Gamma(\overline{I^{n+1}}: \fm^{\infty})\cap \conv(I^{n})\setminus \conv(I^{n+1})\big)\Big)\\
&\ls \#\Big(\ZZ^d\cap \big(\conv(I^{n+1})\cup \cone(I^{n+1})\big)\cap \conv(I^{n})\setminus \conv(I^{n+1})\Big)
\end{align*}
where the last inequality holds by Lemma~\ref{satincone}.  
By \cite[Corollary~3.4]{Sin}, we have $n\conv(I)=\conv(I^n)$, $n\cone(I)=\cone(I^n)$, and $n\bd(I)=\bd(I^n)$ for every $n\gs 1$. Note that $\bd(I)=\conv(I)\cap\cone(I)$; then, 
\begin{equation*}\label{eq:convconvconen}
\begin{aligned}
&(\cone(I^{n+1})\cap \conv(I^{n}))\setminus \conv(I^{n+1})\\
&=
 \big((n+1)\cone(I)\cap n\conv(I)\big)\setminus (n+1)\conv(I)\\
 &=\big((n+1)\cone(I)\cap n\conv(I)\big) \setminus (n+1)\bd(I)\\
 &= \big((n+1)\cone(I)\setminus (n+1)\bd(I)\big)\setminus \big(n\cone(I)\setminus n\bd(I)\big)=\conen{n}(I)\,.
\end{aligned}
\end{equation*}
It follows that $\lambda_R\big(\HH{0}{\fm}{\overline{I^{n}}/\overline{I^{n+1}}}\big)\ls \#(\ZZ^d\cap\conen{n}(I)).$
Therefore, by Proposition~\ref{poly}, part b), 
\[j(I)=\lim_{n\rightarrow\infty}\frac{(d-1)!}{n^{d-1}}\lambda_R\big(\HH{0}{\fm}{\overline{I^{n}}/\overline{I^{n+1}}}\big)\ls d!\vol(\cone(I))\,.\]

\subsection*{Step 2:} The proof that $j(I)\gs d!\vol(\cone(I))$:
\subsection*{Step 2a:} Reduction to the case of an ideal corresponding to a single facet:\\
\indent First we claim that it suffices to verify the inequality for a monomial ideal whose Newton polyhedron has a single bounded facet. Indeed, if the inequality holds for such ideals, write $J_1,\dots, J_t \subset I$ for the monomial ideals corresponding to the bounded facets of $I$ and $\mcal{F}_1,\dots,\mcal{F}_t$ for the corresponding facets, so that we have $\bd(I)=\bigcup_i \mcal{F}_i$ and $\bd(J_i)=\mcal{F}_i$. Then 
since we have 

\[\bigcup_{i=1}^t\Gamma\big((\overline{J_i^{n+1}}:_{\,\overline{J_i^n}}\fm^{\infty})/\overline{J_i^{n+1}}\big)\ \subseteq \  \Gamma\big((\overline{I^{n+1}}:_{\,\overline{I^n}}\fm^{\infty})/\overline{I^{n+1}}\big)\,,\]
and 
\[\bigcup_{i=1}^t \conen{n}(\mcal{F}_i)=\conen{n}(I)\,,\]
we have a containment 
\begin{align*}
\conen{n}(I)\setminus\Gamma\big((\overline{I^{n+1}}:_{\,\overline{I^n}}\fm^{\infty})/\overline{I^{n+1}}\big) &\subseteq \bigcup_{i=1}^t \conen{n}(\mcal{F}_i)\setminus \bigcup_{i=1}^t\Gamma\big((\overline{J_i^{n+1}}:_{\,\overline{J_i^n}}\fm^{\infty})/\overline{J_i^{n+1}}\big)\\
&\subseteq \bigcup_{i=1}^t (\conen{n}(\mcal{F}_i)\setminus\Gamma\big((\overline{J_i^{n+1}}:_{\,\overline{J_i^n}}\fm^{\infty})/\overline{J_i^{n+1}})\big)\,.
\end{align*}
Notice $\#\Gamma\big((\overline{J^{n+1}}:_{\,\overline{J^n}}\fm^{\infty})/\overline{J^{n+1}}\big)=\lambda_R\big(\HH{0}{\fm}{\overline{J^n}/\overline{J^{n+1}}}\big)$ for any ideal $J$, so
\begin{align*}
\# (\conen{n}(I)\cap \ZZ^d)& - \lambda_R\big(\HH{0}{\fm}{\overline{I^n}/\overline{I^{n+1}}}\big)\\
&\ls \sum_{i=1}^t \Big(\# (\conen{n}(\mcal{F}_i)\cap \ZZ^d) - \lambda_R\big(\HH{0}{\fm}{\overline{J_i^n}/\overline{J_i^{n+1}}}\big)\Big)\,.
\end{align*}
Thus, if the claimed inequality holds for each $J_i$, we have that
\begin{align*}
& \lim_{n\rightarrow\infty}\frac{(d-1)!}{n^{d-1}}\Big(\# \big(\conen{n}(I)\cap \ZZ^d\big) - \lambda_R\big(\HH{0}{\fm}{\overline{I^n}/\overline{I^{n+1}}}\big)\Big)\\
&\ls \sum_i \left(\lim_{n\rightarrow\infty}\frac{(d-1)!}{n^{d-1}} \# \big(\conen{n}(\mcal{F}_i)\cap \ZZ^d \big) - \lim_{n\rightarrow\infty}\frac{(d-1)!}{n^{d-1}} \lambda_R\big(\HH{0}{\fm}{\overline{J_i^n}/\overline{J_i^{n+1}}}\big)\right)\\
&=\sum_i\big(d!\vol(\cone(J_i))-j(J_i)\big)\ls 0\,.
\end{align*}
It follows that 
\begin{align*}
 j(I)&=\lim_{n\rightarrow\infty}\frac{(d-1)!}{n^{d-1}}\lambda_R\big(\HH{0}{\fm}{\overline{I^n}/\overline{I^{n+1}}}\big)\\
&\gs\lim_{n\rightarrow\infty}\frac{(d-1)!}{n^{d-1}}\# \big(\conen{n}(I)\cap \ZZ^d\big)=d!\vol(\cone(I))\,,
\end{align*}
where the last equality follows from Proposition~\ref{poly}, part b).

We subsequently assume that the Newton polyhedron of $I$ has a single bounded facet $\tope$.

\subsection*{Step 2b:} Description of a rational polytope containing points contributing to $j(I)$:

Let $\tope$ be $\bd(I)$. If $\dim(\tope)<d-1$, then $\dim(\pyr(\tope))<d$, so $\vol(\pyr(\tope))=0$, and there is nothing to show. Assume $\dim(\tope)=d-1$, and let $H$ be the affine $(d-1)$-plane spanned by $\tope$. Let $\langle x,b \rangle = c$ be a defining equation for $H$. Recall that each entry of $b$ is positive for a bounded face, so that after rescaling $b$, we may assume that $c=1$, with each $b_j>0$.\\

\indent We now describe a region $\mcal{R}_n\subset\conen{n}(\tope)$ such that for any $\alpha\in\ZZ^d\cap\mcal{R}_n$, one has $x^{\alpha}\in
(\overline{I^{n+1}}:_{\,\overline{I^n}}\fm^{\infty})$. 
Note that $x^{\alpha}\in
(\overline{I^{n+1}}:_{R} x_i^{\infty})$ if and only if $\alpha$ is in the image of $\pi_i$, the projection in the $e_i$ direction onto $\langle \alpha, b \rangle H$. That is, 
\[
\Gamma(\overline{I^{n+1}}:_{R} x_i^{\infty})\cap \langle \alpha, b \rangle H=\ZZ^d\cap \pi_i\big((n+1)\tope\big)\,.\]
Let $\tope=\conv(v_1,\dots,v_t)$. Then 
\[
\pi_i\big((n+1) v_j\big)=(n+1) v_j-\frac{(n+1)-\langle \alpha, b \rangle}{\langle e_i, b \rangle}e_i\,,
\]
so that 
\[
\pi_i\big((n+1)\tope\big)=\conv\!\Big((n+1) v_1-\frac{(n+1)-\langle \alpha, b \rangle}{\langle e_i, b \rangle}e_i,\dots,(n+1) v_t-\frac{(n+1)-\langle \alpha, b \rangle}{\langle e_i, b \rangle}e_i\Big)\,.
\]
Since each $b_j>0$, this is well-defined. Now, 
\[
(\overline{I^{n+1}}:_{\,\overline{I^n}}\fm^{\infty})=\overline{I^n}\cap\bigcap_{i=1}^d(\overline{I^{n+1}}:_{R} x_i^{\infty})\,.
\]
We define a region
\[\mcal{R}_n:=\conen{n}(\tope)\cap\bigcap_{i=1}^d \big((n+1)\tope+\RR_{\ls 0}\, e_i \big)\]
so that, by the above, $\mcal{R}_n\cap \ZZ^d$ is contained in $\Gamma\big((\overline{I^{n+1}}:_{\,\overline{I^n}}\fm^{\infty})/\overline{I^{n+1}}\big)$. (In fact, these two sets are quickly verified to be equal, but we will only use the stated containment.)
We remark that for $n\ls s< (n+1)$,
\begin{equation*}\label{rncontain}
\mcal{R}_n\cap sH \ \supseteq \  sH \cap (\mcal{R}_n\cap nH + \RR^d_{\gs 0}) \ \supseteq \
 \frac{s}{n}(\mcal{R}_n \cap nH) \ = \ \frac{s}{n}(\mcal{R}_n \cap n\tope)
\end{equation*}
where the second containment holds because the vectors in $\mcal{R}_n\cap nH$ have nonnegative components. If we set $\tau_{n}=\frac{1}{n}(\mcal{R}_n \cap n\tope)$, then $\conen{n}(\tau_{n})\subset \mcal{R}_n$.
Note that $\tau_{n}$ has vertices in $\QQ^d$. 

We also claim that $\tau_{n}\subset \tau_{n+1}$. We have $\alpha\in\tau_{n}$ if and only if
\[
\frac{n}{n+1}\alpha+\frac{1}{n+1}\frac{1}{\langle e_i, b \rangle}e_i\in \tope
\]
 for all $i$. By convexity of $\tope$ since $\alpha\in\tope$, if $\lambda \alpha + (1-\lambda) \frac{1}{\langle e_i, b \rangle}e_i \in \tope$ then also $\lambda' \alpha + (1-\lambda') \frac{1}{\langle e_i, b \rangle}e_i \in \tope$ for $0\ls\lambda\ls \lambda' \ls 1$. In particular, 
\[
\frac{n+1}{n+2}\alpha+\frac{1}{n+2}\frac{1}{\langle e_i, b \rangle}e_i\in \tope
\]
for all $i$, so that $\alpha\in\tau_{n+1}$. It follows by induction that $\tau_{n}\subset \tau_{n'}$ for $n\ls n'$.

\subsection*{Step 2c:} Using $\cone(\tau_{n})$ to give a lower bound:\\
\indent Consider the distance between a vertex of $n\tope$ and the corresponding vertex of $\pi_i\big((n+1)\tope\big)$. We compute this distance as 
\[
|nv_j-((n+1) v_j - \frac{1}{\langle e_i, b \rangle}e_i)|=|v_j-\frac{1}{\langle e_i, b \rangle}e_i|\,,
\]
which is bounded above uniformly in $n$ by $L:=\max_{i,j}\{|v_j-\frac{1}{\langle e_i, b \rangle}e_i|\}$. Then the region
${n\tau_{n}=\mcal{R}_n \cap n\tope}$ is the intersection of $(d+1)$ many polytopes, each of which is the convex hull of $t$ points $v'_1,\dots,v'_t$ such that $|v_j-v'_j|< L$ for all $j$. That is, each such polytope is an $L$-shaking of $n\tope$ in the sense of Lemma~\ref{close}. Dividing through by $n$ we see that $\tau_{n}$ is the intersection of $(d+1)$ many polytopes that are all $\frac{L}{n}$-shakings of the polytope $\tope$.

Given $0<c<1$, we may now apply Lemma~\ref{close} in the affine subspace $H$. We obtain, for a sufficiently large $M$, a $\tau_M$ such that $\vol (\tau_M) \gs c \vol (\tope)$ in $H$, and hence $\vol (\cone(\tau_M)) \gs c \vol (\cone(\tope))$.

For $n>M$, we have from the previous step that $\tau_n\supseteq\tau_M$, so
\[
\conen{n}(\tau_M)\subseteq\conen{n}(\tau_n)\subseteq \mcal{R}_n\,.
\]
Thus, if $\alpha\in \ZZ^d \cap \conen{n}(\tau_{M})$, then
\[
x^{\alpha} \in (\overline{I^{n+1}}:_{\,\overline{I^n}}\fm^{\infty})/\overline{I^{n+1}}=\HH{0}{\fm}{\overline{I^n}/\overline{I^{n+1}}}\,.
\]
Thus,
\begin{align*}
j(I) &=\lim_{n\rightarrow\infty}\frac{(d-1)!}{n^{d-1}}\lambda_R\big(\HH{0}{\fm}{\overline{I^n}/{\overline{I^{n+1}}}}\big)\\
&\gs \lim_{n\rightarrow\infty}\frac{(d-1)!}{n^{d-1}} \#\big(\ZZ^d \cap \conen{n}(\tau_{M})\big)=d!\vol (\cone(\tau_M))\,,
\end{align*}
where the last equality follows from Proposition~\ref{poly}. Therefore, for all $c<1$, we have the inequality $j(I)\gs c (d!\vol (\cone(\tope)))$, so 
\[j(I)\gs d!\vol (\cone(\tope))=d!\vol (\cone(I))\,,\]
as required.
\end{proof}

\begin{rmk}
If $I$ is an $\fm$-primary monomial ideal, the $j$-multiplicity is equal to the Hilbert-Samuel multiplicity, and $\cone(I)$ is the complement of the Newton polyhedron in $\RR_{\gs 0}^d$. In this way, Theorem~\ref{main} agrees with Teissier's result for $\fm$-primary monomial ideals. 
\end{rmk}

\section{Normal Affine Semigroup Rings}

We next record that, with slight modifications, we can use the same proof to establish a similar result for a wider class of rings. In this section, we state this generalization, and describe necessary modifications to the argument. 

By an \emph{affine semigroup ring}, we mean a ring $A=k[Q]$ that has as a $k$ vector space basis $\{x^q\,|\,q\in Q\}$, where $Q$ is a subsemigroup of $\ZZ^d$ (with the operation~$+$), and multiplication given by $x^{q_1} x^{q_2}=x^{q_1+q_2}$. We denote by $\fm_A$ its maximal homogeneous ideal. If $A$ is a normal ring, then there is a cone $\sigma\subseteq \RR^d$ with finitely many extremal rays, each of which contains a lattice point ($\sigma$ is a \emph{rational cone}) and such that $A\iso k[\ZZ^d \cap \sigma]$, see \cite[Chapters~7~and~10]{MS}. We assume henceforth that $A$ is presented in this form. Additionally, suppose that $\sigma$ is \emph{pointed}, i.e., that it contains no nontrivial linear subspace of $\RR^d$, and that $\dim(\sigma)=d$.

Set $r_1,\dots,r_s$ to be \emph{ray generators} for $\sigma$, i.e., minimal lattice points along the extremal rays of $\sigma$. Let $I$ be a monomial ideal of $R$ minimally generated by monomials $x^{v_1},x^{v_2},\ldots, x^{v_n}$. In this context, we define
\[\conv(I):= \conv(v_1,\ldots, v_n)+\sigma\,,\]
and, as in section two, $H_i=\{x\in \mathbb{R}^n\mid \langle x,b_i\rangle = c_i\}$, with $b_i\in \mathbb{Q}^d$, $c_i\in \mathbb{Q}$ for $i=1,\ldots, \nr$ to be the supporting hyperplanes of $\conv(I)$ so that 
\[\conv(I)= H_{1}^+\cap H_{2}^+\cap\cdots \cap H_{\nr}^+ \, ,\] 
where $H_i^+=\{x\in \mathbb{R}^n\mid \langle x,b_i\rangle \gs c_i\}$. We again assume that $H_1,\ldots,H_u$, are the hyperplanes corresponding to unbounded facets. We retain the other definitions, e.g., $\cone$ and $\conen{n}$, from the preliminary section. Note that since we have chosen an embedding of our semigroup in $\ZZ^d\subset \RR^d$, it makes sense to talk about volume. Our main result from the previous section holds in this context:

\begin{thm}\label{toric} Let $A=k[\ZZ^d \cap \sigma]_{\fm_A}$, where $\sigma$ is a $d$-dimensional pointed rational cone. Let $I\subset A$ be a monomial ideal. Then $j(I)=d!\vol(\cone(I))$.
\end{thm}
\begin{proof}
The proof of Theorem~\ref{main} applies, after some slight changes:

It is easy to see that the inequalities $\langle r_j, b_i \rangle \gs 0$ hold for all $i,j$, and $\langle r_j, b_i \rangle > 0$ for $u+1 \ls i \ls \nr$ and all $j$, as in \cite[Lemma~1.1]{Sin}. Apply these inequalities in the proof of Lemma~\ref{satincone} mutatis mutandis to obtain the same conclusion. Step 1 in the proof of Theorem~\ref{main} follows.

To prove the other inequality in this setting, first note that for any monomial ideal $J$, we have $\alpha\in\Gamma(J :_A \fm_A^{\infty})$ if and only if there exist $a_j\in \RR_{\gs 0}$ such that $\alpha+a_j r_j\in \Gamma(J)$ for all $j$. Thus, for a monomial ideal $I$ with a single bounded face $\tope$, we may define the region
\[ \mcal{R}_n:=\conen{n}(\tope)\cap\bigcap_{j=1}^m \big((n+1)\tope+\RR_{\ls 0}\, r_j \big)\]
so that $\Gamma\big((\overline{I^{n+1}}:_{\,\overline{I^n}}\fm^{\infty})/\overline{I^{n+1}}\big)=\ZZ^d\cap \mcal{R}_n$. This region has all of the salient properties from the case of the polynomial ring $R$  (i.e., $\mcal{R}_n$ is a rational polytope such that $\mcal{R}_n \cap s\tope \supseteq \frac{s}{n}(\mcal{R}_n\cap n\tope)$ for  $n\ls s < (n+1)$), and we employ this to complete the proof of Step 2 as before.
\end{proof}

\begin{rmk}
The $j$-multiplicity of an ideal $I$ is greater than zero if and only if $I$ has maximal analytic spread (\cite[Lemma~3.1]{NU}). Then it follows from Theorems~\ref{main} and \ref{toric} that a monomial ideal of a normal affine semigroup ring has maximal analytic spread if and only if $\conv(I)$ has some bounded facet of dimension $d-1$. In the polynomial case, this fact is easily deduced from the characterization of analytic spread of monomial ideals stated in section~2.
\end{rmk}

\section{The $\varepsilon$-multiplicity as a volume}

In this section we follow the notation from the third section, i.e., \linebreak $R=k[x_1,\ldots,x_d]_{\fm}$. The $\varepsilon$-multiplicity was defined by Ulrich and Validashti \cite{UV} in 2011 as a generalization of the Buchsbaum-Rim multiplicity for submodules of free modules with arbitrary colength. In its simpler form for ideals, the $\varepsilon$-multiplicity is defined by 
\[
\varepsilon(I)=\limsup_{n\rightarrow \infty}\frac{d!}{n^d}\lambda_R(\HH{0}{\fm}{R/I^n})\,,
\]
where the limit of the sequence has been shown to exist in wide generality, see \cite{Cut}.
For monomial ideals, the limit is known to exist and is a rational number as shown in \cite[Corollary~2.5]{HPV}. Nevertheless, unlike the $j$-multiplicity, there are examples of ideals for which the $\varepsilon$-multiplicity is not an integer; see \cite[Example~2.4]{CHS}. In this section, we will give a combinatorial proof of the existence and rationality of the limit in the monomial case, identifying $\varepsilon(I)$ with the normalized volume of a region with rational vertices.

\

Let $\out(I)$ be the region $(H_1^+\cap\cdots\cap H_u^+)\cap (H_{u+1}^-\cup\cdots\cup H_{\nr}^-)$ if the ideal $I$ has maximal analytic spread, and empty otherwise. Notice that by the proof of Lemma~\ref{satincone} we have $H_1^+\cap\cdots\cap H_u^+\subset \conv(I)\cup \pyr(I)$ and then $\out(I)$ is contained in $\pyr(I)\cup \bd(I)$ which is bounded.
\begin{thm}
Let $I\subset R$ be a monomial ideal. Then the limit in the definition of $\varepsilon(I)$ exists, and $\varepsilon(I)=d!\vol(\out(I))$, which is a rational number.
\end{thm}

\begin{proof}
Since the functor $\HH{0}{\fm}{-}$ is sub-additive on short exact sequences, we have 
\[
\lambda_R\big(\HH{0}{\fm}{R/I^n}\big)\ls\sum_{i=0}^{n-1} \lambda_R\big(\HH{0}{\fm}{I^i/I^{i+1}}\big).\]
The right hand side is the sum transform of the function that defines the $j$-multiplicity, and hence for $n\gg 0$ it is equal to a polynomial of degree $d$ and leading coefficient $\frac{j(I)}{d!}$, see \cite{NU}. It follows that $\varepsilon(I)\ls j(I)$, so we can assume that $I$ has maximal analytic spread.

\subsection*{Step a:} Existence of the limit for the filtration $\{\overline{I^n}\}_{n\in\NN}$:
\begin{align*}
\Gamma\big((\overline{I^n}:_R\fm^{\infty})/\overline{I^n}\big) &= \Gamma(\overline{I^n}:_R\fm^{\infty})\setminus \Gamma(\overline{I^n})\\
&= \Gamma(H_{1}^+\cap H_{2}^+\cap\cdots \cap H_{u}^+)\setminus \Gamma(H_{1}^+\cap H_{2}^+\cap\cdots \cap H_{w}^+)\\
&= \Gamma((H_1^+\cap\cdots\cap H_u^+)\cap (H_{u+1}^-\cup\cdots\cup H_{\nr}^-))\setminus  \bigcup_{i=u+1}^{w} \Gamma(\mcal{F}_i) 
\end{align*}
where the equality follows from the fact that $H_{1}^+\cap H_{2}^+\cap\cdots \cap H_{w}^+\cap H_{i}^-=\mcal{F}_i$ for every $i$. From this we conclude 
$$\Gamma\big((\overline{I^n}:_R\fm^{\infty})/\overline{I^n}\big)=\ZZ^d\cap \big(\out(I)\setminus \bd(I)\big).$$ 
Let $\mcal{Q}=\conv(\ver(I))$ and $\mcal{Q}'=\conv(\mcal{Q},\out(I))$; it is easy to check that there is an equality $\mcal{Q}'\setminus \mcal{Q}=\out(I)\setminus \bd(I)$. By \cite[Lemma~3.3]{Sin}, the hyperplanes $\{nH_i\}$, for $1\ls i\ls {\nr}$, are the supporting hyperplanes of $\conv(I^n)$ for each $n\gs 1$. Then we also have $n\mcal{Q}'\setminus n\mcal{Q}=\out(I^n)\setminus \bd(I^n)$.

Hence, 
\begin{align*}\lambda_R\big(\HH{0}{\fm}{R/\overline{I^n}}\big)&=\#\big(\ZZ^d\cap(\out(I^n)\setminus \bd(I^n))\big)\\
&=\#\big(\ZZ^d\cap(n\mcal{Q}'\setminus n\mcal{Q})\big)\\
&=E_{\mcal{Q}'}(n)-E_{\mcal{Q}}(n)\,,
\end{align*}
where the latter is the difference of two Ehrhart quasi-polynomials of the form 
\[
\big(\vol(\mcal{Q}')-\vol(\mcal{Q})\big)n^d+O(n^{d-1})=\vol(\out(I))n^d+O(n^{d-1})\]
 (see proof of Lemma~\ref{poly}), and the result follows.

\subsection*{Step b:} Existence of the original limit:

By \cite[Theorem~7.29]{Vas}, $\overline{I^{n+1}}=I\overline{I^n}$ for $n\gs d$. Then we have the following exact sequences for $n\gs 0$:

\[0\rightarrow \overline{I^n}/I^n \rightarrow R/I^n \rightarrow R/\overline{I^n} \rightarrow 0 \,,\]

\[0\rightarrow I^n/\overline{I^{n+d}} \rightarrow R/\overline{I^{n+d}} \rightarrow R/I^n \rightarrow 0\,, \]
and the following inequalities

\begin{align}
\lambda_R\big(\HH{0}{\fm}{R/\overline{I^{n+d}}}\big) &\ls  \lambda_R\big(\HH{0}{\fm}{R/I^n}\big)\!+\!\lambda_R\big(\HH{0}{\fm}{I^n/\overline{I^{n+d}}}\big) \label{ineq2}\\
&\ls \lambda_R\big(\HH{0}{\fm}{R/\overline{I^n}}\big)\!+\!\lambda_R\big(\HH{0}{\fm}{\overline{I^n}/I^n}\big)\!+\!\lambda_R\big(\HH{0}{\fm}{I^n/\overline{I^{n+d}}}\big)\label{ineq3}.
\end{align}
Now, $\lambda_R\big(\HH{0}{\fm}{I^n/\overline{I^{n+d}}}\big)\ls \lambda_R\big(\HH{0}{\fm}{\overline{I^n}/\overline{I^{n+d}}}\big)\ls \sum_{i=0}^{d-1} \lambda_R\big(\HH{0}{\fm}{\overline{I^{n+i}}/\overline{I^{n+i+1}}}\big)$,  so  
\[
\limsup_{n\rightarrow \infty}\frac{d!}{n^d}\lambda_R\big(\HH{0}{\fm}{I^n/\overline{I^{n+d}}}\big)\ls \sum_{i=0}^{d-1}\limsup_{n\rightarrow \infty} \frac{d!}{n^d}\lambda_R\big(\HH{0}{\fm}{\overline{I^{n+i}}/\overline{I^{n+i+1}}}\big)=0\,,\]
 where the last equality holds by Proposition~\ref{jbar}. 

\noindent Therefore, $\lim _{n\rightarrow \infty}\frac{d!}{n^d}\lambda_R\big(\HH{0}{\fm}{I^n/\overline{I^{n+d}}}\big)=0$.

\noindent Similarly, for $n\gs d$, 
\[\lambda_R\big(\HH{0}{\fm}{\overline{I^n}/I^n}\big)\ls \lambda_R\big(\HH{0}{\fm}{I^{n-d}/I^n}\big)\ls \sum_{i=0}^{d-1}\lambda_R\big(\HH{0}{\fm}{I^{n-d+i}/I^{n-d+i+1}}\big)\,,\]
and then $\lim _{n\rightarrow \infty}\frac{d!}{n^d}\lambda_R \big(\HH{0}{\fm}{\overline{I^n}/I^{n}}\big)=0$. 

\noindent Using these two limits in $\eqref{ineq2}$ and $\eqref{ineq3}$, we obtain
\begin{align*}
d!\vol(\out(I))&=\liminf_{n\rightarrow \infty}\frac{d!}{n^d}\lambda_R\big(\HH{0}{\fm}{R/\overline{I^{n+d}}}\big)\\
 &\ls \liminf_{n\rightarrow \infty}\frac{d!}{n^d}\lambda_R\big(\HH{0}{\fm}{R/I^n}\big)\\
 &\ls \limsup_{n\rightarrow \infty}\frac{d!}{n^d}\lambda_R\big(\HH{0}{\fm}{R/I^n}\big)\\
 &\ls \limsup_{n\rightarrow \infty}\frac{d!}{n^d}\lambda_R\big(\HH{0}{\fm}{R/\overline{I^n}}\big)= d!\vol(\out(I))
\end{align*}
which finishes the proof.
\end{proof}

\section{Examples} 
\begin{figure}
\begin{minipage}[b]{1\linewidth}
\centering
\includegraphics[scale=.5]{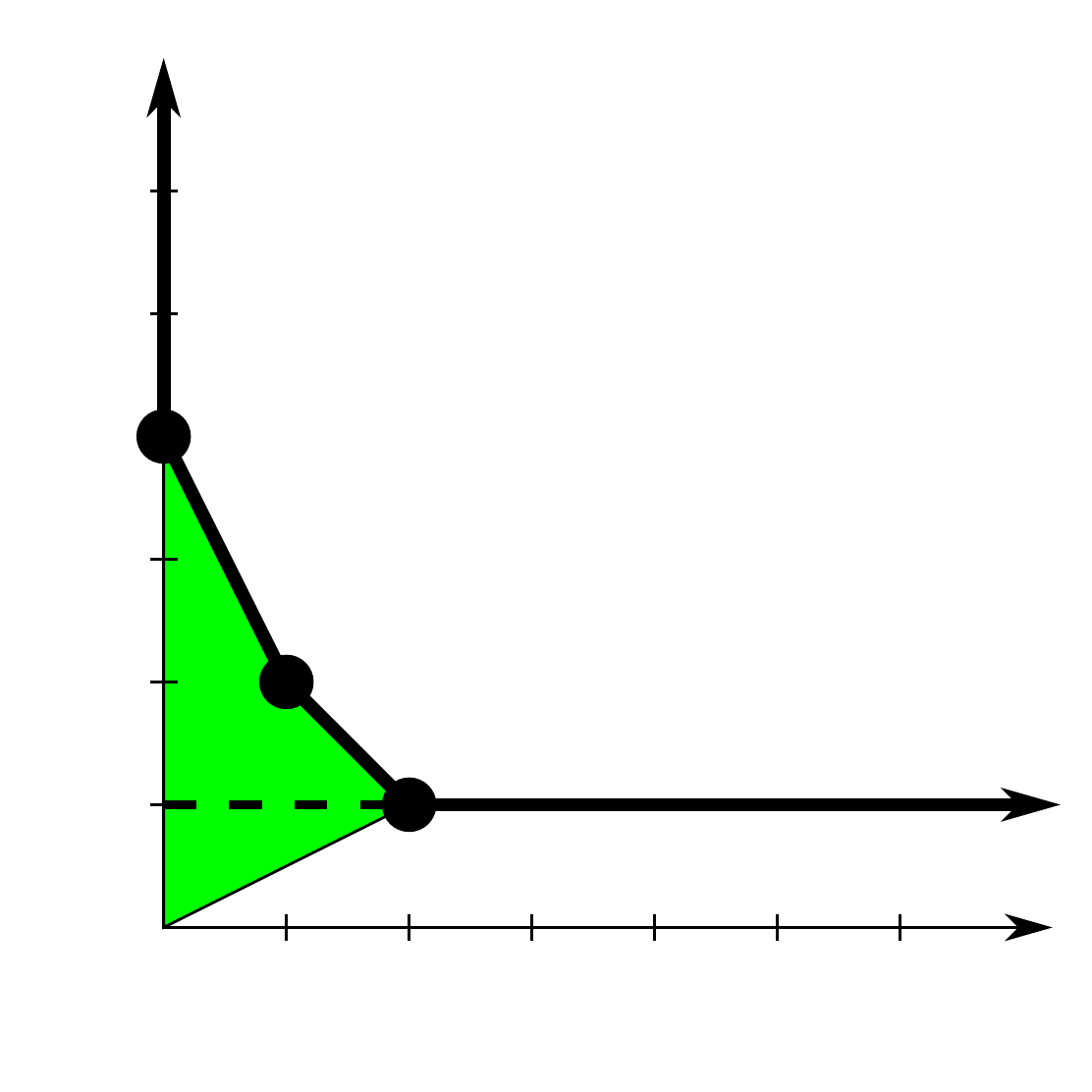}
\caption{Regions for the ideal $(y^4,x^2y,xy^2)$}
\label{fig2}
\end{minipage}\\
\begin{minipage}[b]{1\linewidth}
\centering
\includegraphics[scale=.4]{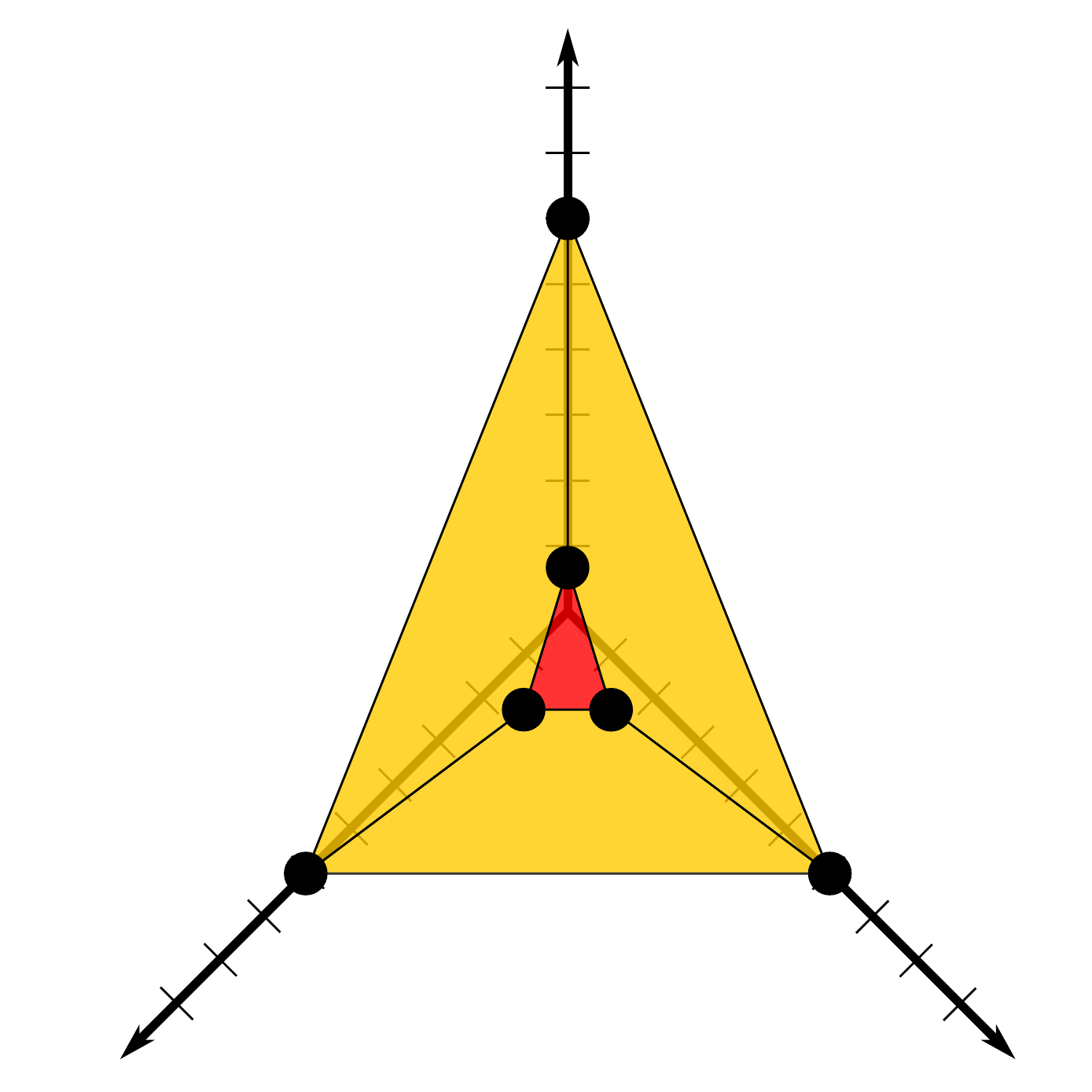}
\caption{Regions for the ideal $(x^6,y^6,z^6,x^2yz,xy^2z,xyz^2)$}
\label{fig3}
\end{minipage}
\end{figure}

\begin{xmp}Let $I=(y^4,x^2y,xy^2)$. We compute the $j$-multiplicity as two times the area of the green region in Figure~\ref{fig2}, obtaining $j(I)=7$. For the $\varepsilon$-multiplicity, we take two times the area of the portion of the green region that lies above the dotted line, obtaining $\varepsilon(I)=5$.
\end{xmp}

\begin{xmp} Let $I=(x^6,y^6,z^6,x^2yz,xy^2z,xyz^2)$. This monomial ideal is $\fm$-primary, so we can compute its Hilbert-Samuel multiplicity as the volume underneath the bounded faces of its Newton polyhedron, which is depicted in Figure~\ref{fig3}. We decompose this region as three regions under the yellow faces corresponding to monomial ideals ${I_1=(x^6,y^6,x^2yz,xy^2z)}$, $I_2=(y^6,z^6,xy^2z,xyz^2)$, $I_3=(x^6,z^6,x^2yz,xyz^2)$, and one region under the red face corresponding to $I_4=(x^2yz,xy^2z,xyz^2)$. We compute $j(I_1)=j(I_2)=j(I_3)=42$, $j(I_4)=4$, so that $e(I)=j(I)=130$.
\end{xmp}

\begin{xmp} \cite[Example~2.4]{CHS} Let $I=(xy,yz,zx)$. The region $\out(I)$ is the tetrahedron with vertices $\{(1,0,1),(1,1,0),(0,1,1),(\frac12,\frac12,\frac12)\}$, and its volume is $\frac{1}{12}$. Thus, $\varepsilon(I)=\frac12$.
\end{xmp}

\begin{xmp} For a graph $G$ on the vertex set $\{1,\dots,d\}$, the \emph{edge ideal} of $G$ is the monomial ideal 
\[ I_G:= (x_i x_j \,|\, \{i,j\}\, \text{is an edge of} \ G) \subset k[x_1,\dots,x_d] \,. \]
The region $\bd(I_G)$ is known in the literature as the \emph{edge polytope} of $G$, see \cite{OH}. By Theorem~\ref{main}, the $j$-multiplicity $j(I_G)$ is $(d-1)!\cdot h$ times the volume of the edge polytope of $G$, where $h$ is the distance from the origin to the plane $\sum x_i=1$. As a particular example, let $C_d$ be the cycle on $d$ vertices. Then,
\[ j(I_{C_d}) = \begin{cases} 0 &\mbox{if } d \ \mbox{is even} \\
2 & \mbox{if } d \ \mbox{is odd}.\end{cases}
\]

\end{xmp}

\section*{Acknowledgements}
The authors would like to thank their Ph.D. advisors, Anurag Singh and Bernd Ulrich. Additionally, we thank Alexandra Seceleanu for her help in finding an appropriate reference for Lemma~$\ref{close}$, and for many helpful comments on an earlier draft of this paper. Both authors also wish to thank Bernd Sturmfels for his enthusiastic instruction on monomial ideals. We are also grateful to the referee for his or her thorough and helpful corrections. This material is based upon the work supported by the National Science Foundation under Grant No. 0932078 000, while the authors were in residence at the Mathematical Sciences Research Institute in Berkeley, California, during the Fall semester of 2012. The first author was also supported in part by NSF grants DMS 0758474 and DMS 1162585, and the second author by NSF grant DMS 0901613.

\end{document}